\documentclass[11pt]{article}
\usepackage{fancyhdr,graphicx}
\usepackage{amsfonts}
\usepackage{amssymb}
\usepackage{amsthm}
\usepackage{newlfont}
\usepackage{amsmath} 
\usepackage[frame,arrow,curve,matrix]{xy}
\usepackage{CJK}

\newtheorem{thm}{Theorem}[section]
\newtheorem{Lemma}[thm]{Lemma}

\newtheorem{definition}[thm]{Definition}

\newtheorem{remark}[thm]{Remark}

\usepackage{mathrsfs}  
\usepackage{subfig}
\usepackage{picinpar}
\usepackage{fancybox}
\begin{document}

\title{On arrow polynomials of checkerboard colorable virtual links}

\author{Qingying Deng\thanks{School of Mathematics and Computational Science, Xiangtan University, Xiangtan, Hunan 411105,
P. R. China}, Xian'an Jin\thanks{School of Mathematical Sciences, Xiamen University, Xiamen, Fujian 361005,
P. R. China. E-mail:xajin@xmu.edu.cn (X. Jin)}, Louis H. Kauffman\thanks{Department of Mathematics, Statistics, and Computer Science, University of Illinois at Chicago, Chicago 60607, USA
and Department of Mechanics and Mathematics, Novosibirsk State University,Novosibirsk, Russia }}


\maketitle

\begin{abstract}
In this paper we give two new criteria of detecting the checkerboard colorability of virtual links by using odd writhe and arrow polynomial of virtual links, respectively. By applying new criteria, we prove that 6 virtual knots are not checkerboard colorable, leaving only one virtual knot whose checkerboard colorability is unknown among all virtual knots up to four
classical crossings.
\end{abstract}

$\mathbf{keywords:}$ Virtual link; checkerboard colorability; odd writhe; arrow polynomial.

\vskip0.5cm

\section{Introduction}
\noindent

In 1996, Kauffman introduced the notion of a virtual knot, which is a generalization
of a classical knot \cite{VKT}. A (virtual) link is the disjoint union of finite number of (virtual) knots.
A \emph{virtual link diagram} is a closed 1-manifold generically
immersed in $\mathbb{R}^{2}$ such that each double point has information of a classical
crossing or a virtual crossing which is indicated by a small circle around the double
point. A \emph{virtual link} is the equivalence class of a virtual link diagram under the
generalized Reidemeister moves including three Reidemeister moves and four virtual Reidemeister moves as illustrated in Fig. \ref{F:classicalRM} and Fig. \ref{RM}, respectively.
Note that the virtual Reidemeister moves are equivalent to the detour move as shown in Fig. \ref{dm}.

\begin{figure}[!htbp]
  \centering
  \includegraphics[width=5.5cm]{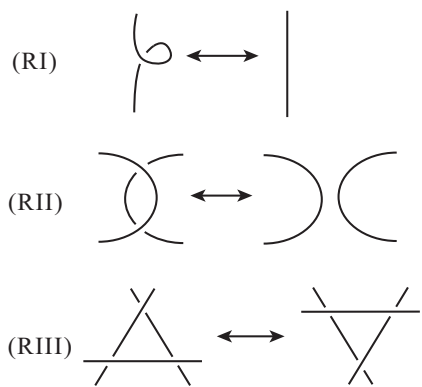}
   \renewcommand{\figurename}{Fig.}
  \caption{Reidemeister moves.}
  \label{F:classicalRM}
\end{figure}

\begin{figure}[!htbp]
  \centering
  \includegraphics[width=6cm]{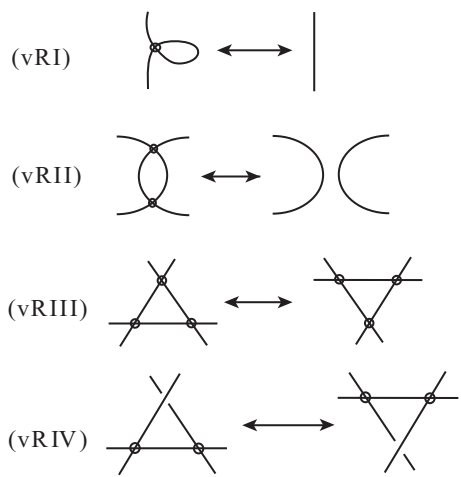}
   \renewcommand{\figurename}{Fig.}
  \caption{Virtual Reidemeister moves.}
  \label{RM}
\end{figure}

\begin{figure}[!htbp]
  \centering
  \includegraphics[width=8cm]{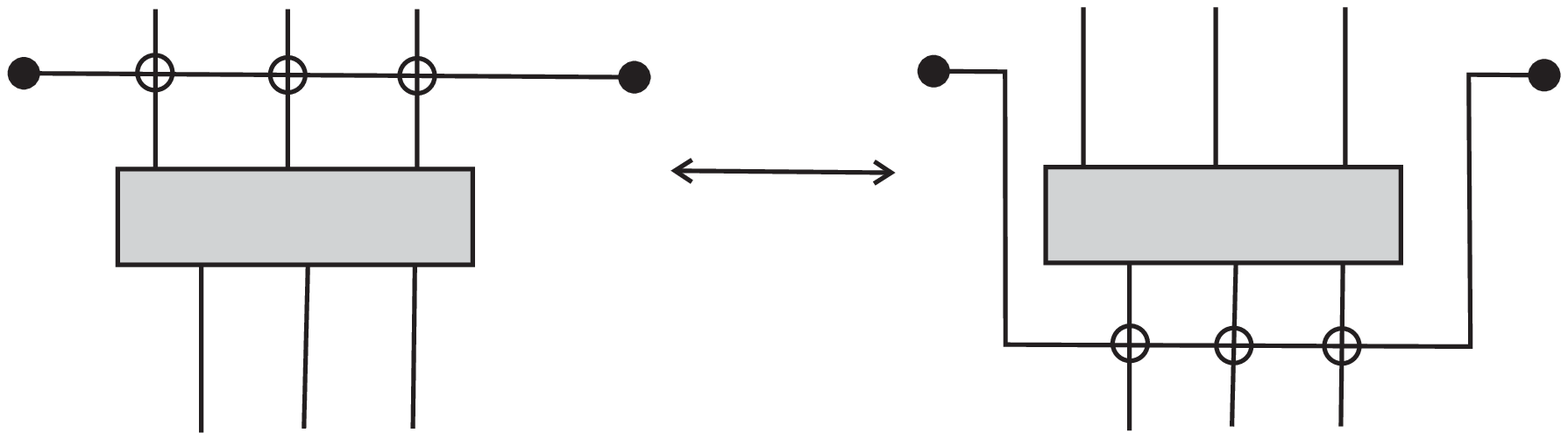}
   \renewcommand{\figurename}{Fig.}
  \caption{Detour move.}
  \label{dm}
\end{figure}

The notion of a checkerboard coloring for a virtual link diagram was first introduced
by Kamada in \cite{Kamada,K2} by using the corresponding abstract link diagram defined in \cite{Kam}.
A virtual link diagram is said to be $\textit{checkerboard colorable}$ if there is a coloring of a small neighbourhood of one side of each arc in the diagram such that near a classical crossing the coloring alternates, and near a virtual crossing the colorings go through independent of the crossing strand and its coloring. Two examples are given in Fig. \ref{ckb}.
A virtual link is said to be $\textit{checkerboard colorable}$ if it has a checkerboard colorable diagram.

\begin{figure}[!htbp]
  \centering
  \includegraphics[width=3.5in]{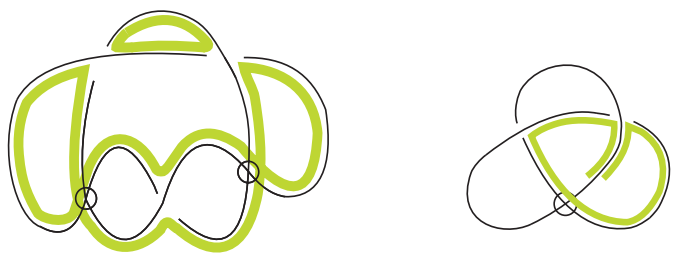}
  \renewcommand{\figurename}{Fig.}
\caption{The left virtual link diagram is checkerboard colorable and the right one is not.}\label{ckb}
\end{figure}

For a virtual link diagram $D$, let $|D|$ be the \emph{underlying 4-valent graph} which is obtained from $D$
by regarding all real crossings as the vertices of $|D|$ and keeping or ignoring virtual crossings.
In \cite{Kamada02} it was observed that giving a checkerboard coloring for $D$ is equivalent to giving
an \emph{alternate orientation} to $|D|$ that is an assignment of orientations to the edges
of $|D|$ satisfying the conditions illustrated in Fig. \ref{Fig.alterorietation}.
Note that every checkerboard colorable virtual link diagram has two kinds of checkerboard colorings and not every virtual link diagram is checkerboard colorable.
Checkerboard colorability of a virtual link diagram is not necessarily preserved by generalized Reidemeister moves.
See the explanation in \cite{Ima}.
Furthermore, note that an alternating virtual link \cite{K2} must be checkerboard colorable (see the explanation in Theorem \ref{main0}), on the contrary, it is not true.

\begin{figure}[!htbp]
  \centering
  \includegraphics[width=4cm]{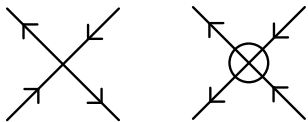}
  \renewcommand{\figurename}{Fig.}
\caption{An alternate orientation.}\label{Fig.alterorietation}
\end{figure}

Sometimes, compared to all virtual links, the properties of classical links are more easily extended to checkerboard colorable virtual links.
For example, Deng, Jin and Kauffman \cite{DengJK} extended the relationship between signed plane graphs and classical link diagrams to signed cyclic graphs and checkerboard colorable virtual link diagrams.
Im, Lee and Lee \cite{ImLeeLee} extended the signature, nullity and determinant of classical oriented links to checkerboard colorable oriented virtual links by presenting the Goeritz matrix for checkerboard colorable virtual links.
Sometimes, although the invariants of classical links can be extended to virtual links completely, say the Jones polynomial of a virtual link \cite{VKT} and the normalized arrow polynomial of a virtual knot \cite{DKArrow}), the invariants of virtual links generally have some features different from those of classical links. But
Kamada \cite{Kamada} proved that Jones polynomials of checkerboard colorable virtual links
have a certain property sharing with those of classical links. Sawollek defined the Sawollek polynomial of a virtual link \cite{Sawollek}, which is in fact the same as the polynomial invariant of virtual links defined by Kauffman and Radford in \cite{KauRad}, and an alternative
method was given by Silver and Williams in \cite{Silver}. Imabeppu \cite{Ima} recently extracted from Sawollek
polynomials two new criteria of checkerboard colorability of virtual links. He applied these two new criteria and the other three well-known methods (introduced by Kamada \cite{Kamada}, Im-Lee-Lee \cite{ImLeeLee} and Im-Lee-Son \cite{ImLeeSon}, respectively) to detect checkerboard colorability of virtual knots up to four classical crossings in the `Virtual Knot Table' \cite{Gr}. There are seven virtual knots whose checkerboard colorability have not been detected.

This paper is organized as follows.
In Section 2, we recall the definition of odd writhe for a virtual knot given by
Kauffman \cite{Kauffman04} and give one new criterion for detecting
checkerboard colorability of virtual links.
In Section 3, we recall the definition of the arrow polynomial for a virtual link given by Dye and
Kauffman \cite{DKArrow}. In Section 4 we give another new criterion for detecting
checkerboard colorability of virtual links in terms of the arrow polynomial. In the final Section 5, by applying new criteria, we prove that 6 virtual knots are not checkerboard colorable, leaving only one virtual knot whose checkerboard colorability is unknown among all virtual knots up to four
classical crossings.

\section{Odd writhes of virtual knots}
\noindent

Let $D$ be an oriented virtual knot diagram. The \emph{writhe} $\omega(D)$ of $D$ is defined as the sum of the signs of classical crossings of $D$. The convention for the signs of classical crossings is shown in Fig. \ref{F:csign}.
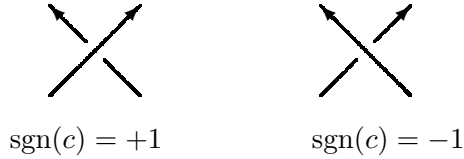
\begin{figure}[!htbp]
  \centering
  \unitlength=0.6mm
  \begin{picture}(0,35)
  \thicklines
  \qbezier(-40,10)(-40,10)(-20,30)
  \qbezier(-40,30)(-40,30)(-32,22)
  \qbezier(-20,10)(-20,10)(-28,18)
  \put(-35,25){\vector(-1,1){5}}
  \put(-25,25){\vector(1,1){5}}
  \put(-32,0){\makebox(0,0)[cc]{$\operatorname{sgn}(c)=+1$}}
  \qbezier(40,10)(40,10)(20,30)
  \qbezier(40,30)(40,30)(32,22)
  \qbezier(20,10)(20,10)(28,18)
  \put(25,25){\vector(-1,1){5}}
  \put(35,25){\vector(1,1){5}}
  \put(35,0){\makebox(0,0)[cc]{$\operatorname{sgn}(c)=-1$}}
  \end{picture}
    \renewcommand{\figurename}{Fig.}
  \caption{Signs of classical crossings.}
  \label{F:csign}
\end{figure}

A crossing $i$ in a virtual knot diagram $D$ is said to be \emph{odd} (resp. \emph{even}) if one encounters an odd (resp. {even}) number
of classical crossings in walking along the diagram on one full path that starts at $i$ and returns to $i$.
Let Odd(D) denote the set of odd crossings of $D$.
In a classical knot diagram $D$, the set Odd(D) is empty.
The \emph{odd writhe} $J(D)$ of $D$ is defined by
\begin{eqnarray}\label{oddwrithe}
J(D) &=& \sum_{c\in Odd(D)}sgn(c).
\end{eqnarray}

It is easy to prove that $J(D)$ is invariant under the generalized Reidemeister moves. We refer the reader to \cite{Kauffman04,Kau3}. For an oriented virtual knot $K$ represented by an oriented knot diagram $D$, let $J(K)=J(D)$, and call it the \emph{odd writhe} of $K$. The invariant $J(K)$ can be used to detect nontriviality, non-classicality and chirality for infinitely many virtual knots \cite{Kauffman04}.
Moreover, it can also be used to detect the checkerboard colourability of a virtual knot as stated in the following theorem.

\begin{thm}\label{main0}
Let $K$ be an oriented virtual knot. If $J(K)\neq 0$, then $K$ is not checkerboard colorable.
\end{thm}

\begin{proof}
Suppose that $K$ is checkerboard colorable and $D$ is a checkerboard colorable diagram of $K$.
Then its underlying 4-valent graph $|D|$ has an alternate orientation.
It is easy to show that each classical crossing of $D$ must be even as shown in Fig. \ref{Fig.evencross}.
Hence $J(D)=0$ and thus $J(K)=0$.
\begin{figure}[pb]
\centering
\includegraphics[width=3.5cm]{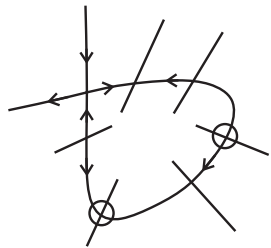}
\renewcommand{\figurename}{Fig.}
\caption{{\footnotesize An alternate orientation of a vertex in $|D|$}}\label{Fig.evencross}
\end{figure}
\end{proof}

We use a table of virtual knots given by Green under the supervision of
Bar-Natan \cite{Gr}.
By computing $J(K)$ of virtual knots with up to four classical crossings, we obtain that virtual knots 2.1, 3.2, 3.3, 3.4, 4.1, 4.3, 4.4, 4.5, 4.7, 4.9, 4.11, 4.14, 4.15,
4.18, 4.20, 4.22, 4.25, 4.27, 4.28, 4.29, 4.30, 4.33, 4.34, 4.37, 4.38, 4.39, 4.40, 4.43, 4.44, 4.45,
4.48, 4.49, 4.52, 4.53, 4.54, 4.60, 4.61, 4.62, 4.63, 4.64, 4.69, 4.73, 4.74, 4.78, 4.80, 4.81, 4.82, 4.83, 4.84, 4.87,
4.88, 4.91, 4.92, 4.94, 4.95, 4.100, 4.101 and 4.104 are not checkerboard colorable.

\section{Arrow polynomials of virtual links}
\noindent

In this section, we recall the arrow polynomial of an oriented virtual link diagram defined by Dye and Kauffman \cite{DKArrow}. Readers familiar with it can skip this section.

Let $D$ be an oriented virtual link diagram. We shall denote the arrow polynomial
by the notation $\langle D\rangle_{A}$.
The \emph{arrow polynomial} of $D$ is based on the oriented state expansion as shown in Fig. \ref{Fig.arrow}.
Note that each classical crossing has two types of smoothing, one \emph{oriented smoothing} and the other \emph{disoriented smoothing}.
Each disoriented smoothing gives rise to a \emph{cusp pair} where each \emph{cusp} is denoted by an angle with arrows either both entering the vertex or both leaving the vertex.
Furthermore, the angle locally divides the plane into two parts: One part is the span of an acute angle (of size less than $\pi$); the other part is the span of an obtuse angle.
The \emph{inside} of the cusp denotes the span of the acute angle.
It is obvious that the total number of cusps of a state circle after smoothing all classical crossing will be even according to the orientations on the edges of the state circle.

\begin{figure}[!htbp]
  \centering
  \includegraphics[width=8cm]{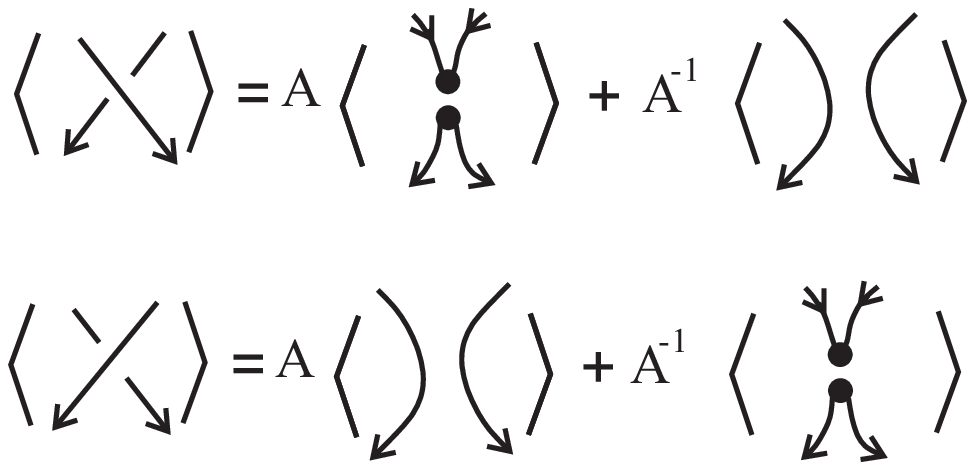}
    \renewcommand{\figurename}{Fig.}
  \caption{Oriented state expansion.}
  \label{Fig.arrow}
\end{figure}

The structure with cusps is reduced according to a set of rules that ensures invariance of the state summation under generalized Reidemeister moves.
The basic convention for this simplification are shown in Fig. \ref{Fig.reduction}.
When the insides of the cusps are on opposite sides of the connecting
segment (a ``zig-zag''), then no cancellation is allowed. Each state circle is seen as a \emph{circle graph} with extra nodes corresponding to the cusps. All graphs are taken up to virtual equivalence, as explained above. Fig. \ref{Fig.reduction} illustrates the simplification of three circle graphs. In one case the graph reduces to a circle with no vertices. In each of the other two cases there is no further cancellation, but the graph is equivalent to one without a virtual crossing. Note that the virtual crossings in a state have the potential to effect the total number of cusps.

\begin{figure}[!htbp]
  \centering
  \includegraphics[width=10cm]{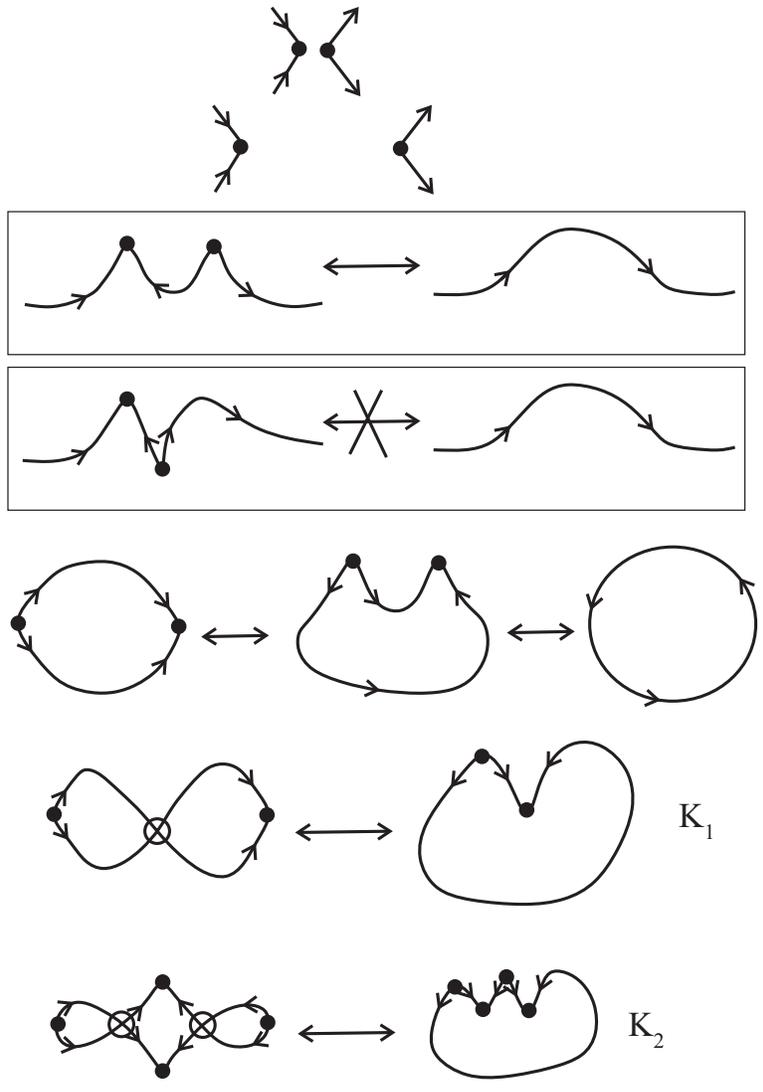}
    \renewcommand{\figurename}{Fig.}
  \caption{Reduction rule for the arrow polynomial.}
  \label{Fig.reduction}
\end{figure}

Use the reduction rule of Fig. \ref{Fig.reduction} so that each state
is a disjoint union of \emph{reduced circle graphs}.
Since such graphs are planar, each is equivalent to
an embedded graph (no virtual crossings) via the detour move. For each reduced circle graph, it has even, say
$2n$ number of vertices that alternate in type around the circle so that $n$ are pointing inward (the angle of the cusp is obtuse in the inside of the circle) and
$n$ are pointing outward (the angle of the cusp is acute in the inside of the circle). The circle with no vertices is evaluated as $d=-A^{2}-A^{-2}$ as is usual for these expansions, and the circle is removed from the graphical expansion.
Let $K_{n}$ denote the
circle graph with $2n$ alternating vertex types as shown in Fig. \ref{Fig.reduction} for $n=1$ and $n=2$. Each
circle graph contributes $d=-A^{2}-A^{-2}$ to the state sum and the graphs $K_{n}$ for $n\geq1$ remain in the graphical expansion. Each $K_{n}$ is an extra variable in the polynomial. Thus a product of the
$K_{n}$'s corresponds to a state that is a disjoint union of copies of these circle graphs.
Note that we continue to use the caveat that an isolated circle or circle graph (i.e. a state consisting in a single circle or single circle graph) is assigned a state circle value of
unity in the state sum. This assures that $\langle D\rangle_{A}$ is normalized so that the
unknot receives the value one.
Note that the arrow polynomial will reduce to the classical bracket polynomial
when each of the new variables $K_{n}$ is set equal to unity.
Formally, we have the following state summation for the arrow polynomial.

\begin{definition}(\cite{DKArrow})
The arrow polynomial $\langle D\rangle_{A}$ of an oriented virtual link diagram $D$ is defined by
\begin{eqnarray}\label{arrowpoly}
  \langle D\rangle_{A} &=& \sum_{S}A^{\alpha-\beta}d^{|S|-1}\langle S\rangle,
\end{eqnarray}
where $S$ runs over the oriented bracket states of the diagram, $\alpha$ denotes the number of smoothings
with coefficient $A$ in the state $S$ and $\beta$ denotes the number with coefficient
$A^{-1}$, $d=-A^{2}-A^{-2}$, $|S|$ is the number of circle graphs in the state, and
$\langle S\rangle$ is a product of extra variables $K_{1},K_{2},...$ associated with the non-trivial circle graphs in the state $S$.
\end{definition}

Note that each circle graph (trivial or not) contributes to the power of
$d$ in the state summation, but only non-trivial circle graphs contribute to
$\langle S\rangle$. Let $D$ be an oriented virtual link diagram with writhe $\omega(D)$.
The normalized version is defined by
\begin{eqnarray}\label{Norarrowpoly}
  \langle D\rangle_{NA} &=& (-A^{3})^{-\omega(D)}\langle D\rangle_{A}.
\end{eqnarray}
For an oriented virtual link $L$ represented by an oriented link diagram $D$, we denote $\langle D\rangle_{NA}$ by $\langle L\rangle_{NA}$, and call it the \emph{arrow polynomial} of $L$.

\begin{thm}(\cite{DKArrow}, Theorem 1.2)
Let $D$ be an oriented virtual link diagram. Then $\langle D\rangle_{A}$ is an
invariant under the Reidemeister moves $II$ and $III$ and virtual Reidemeister
moves.
\end{thm}

\begin{thm}(\cite{DKArrow}, Theorem 1.6)
Let $L$ be an oriented virtual link.
Then $\langle L\rangle_{NA}$ is invariant under the generalized Reidemeister moves.
\end{thm}

We apply the invariant to the virtualized trefoil (the mirror of $3.7$ in \cite{Gr}), denoted by $3.7^{*}$ and pictured in Fig. \ref{Fig.VT}. The (unreduced) states of the virtual trefoil are shown in Fig.
 \ref{Fig.SVT}. Thus,
\begin{eqnarray}
   \langle3.7^{*}\rangle_{NA}&=&-A^{-3}(-A^{-5}+K_1^{2}A^{-5}-K_1^{2}A^{3}).
\end{eqnarray}

\begin{figure}[!htbp]
  \centering
  \includegraphics[width=2.5cm]{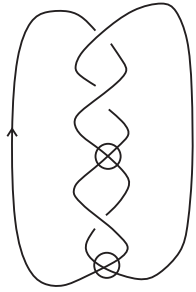}
    \renewcommand{\figurename}{Fig.}
  \caption{Virtualized trefoil $3.7^{*}$.}
  \label{Fig.VT}
\end{figure}

\begin{figure}[!htbp]
  \centering
  \includegraphics[width=14cm]{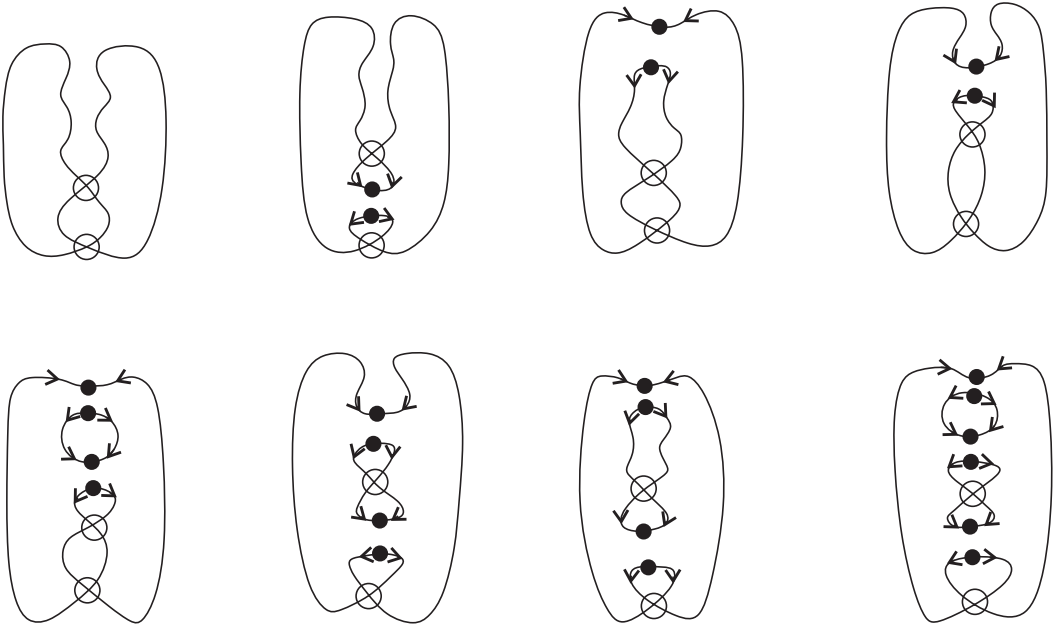}
    \renewcommand{\figurename}{Fig.}
  \caption{States of the virtualized trefoil $3.7^{*}$.}
  \label{Fig.SVT}
\end{figure}

Note that any summand of $\langle D\rangle_{NA}$ has the following form:
\begin{equation}\label{summand}
  A^{s}(K^{j_{1}}_{i_{1}}K^{j_{2}}_{i_{2}}\cdots K^{j_{v}}_{i_{v}}).
\end{equation}
Then the \emph{$k$-degree} of this summand is defined to be
\begin{equation}\label{kdeg}
 i_{1}\times j_{1} + i_{2}\times j_{2} +\cdots + i_{v}\times j_{v},
\end{equation}
which is equal to the half of reduced number of cusps in the state associated
with these variables. Notice that if the summand has no $K_{C}$ variables, then
the $k$-degree is zero.

Let $AS(D)$ denote the set of $k$-degrees obtained from the set of summands of $\langle D\rangle_{NA}$.

\begin{Lemma}(\cite{DKArrow}, Lemma 1.3)
For an oriented virtual link diagram $D$, $AS(D)$ is invariant under the generalized Reidemeister moves.
\end{Lemma}

Dye and Kauffman \cite{DKArrow} showed that the phenomenon of cusped states and extra
variables $K_{n}$ only occurs for virtual links.

\begin{thm}(\cite{DKArrow}, Theorem 1.5)
If $D$ is a classical link diagram then $AS(D)=\{0\}$.
\end{thm}

\section{New criteria from arrow polynomials}
\noindent

We now determine the characteristics of arrow polynomials of checkerboard colorable virtual link diagrams. For this we need the concept of virtual braid.
Kauffman and Lambropoulou \cite{KauLam} gave a general method for converting virtual links to
virtual braids. The braiding method given there is quite general and applies to
all the categories in which braiding can be accomplished. It includes the braiding of
classical, virtual, flat, welded, unrestricted, and singular knots and links.
We just explain the braiding techniques for a classical crossing and a virtual crossing by two schemas as shown in Figs. \ref{Fig.fulltwist} and \ref{Fig.halftwist}.

\begin{figure}[!htbp]
  \centering
  \includegraphics[width=12cm]{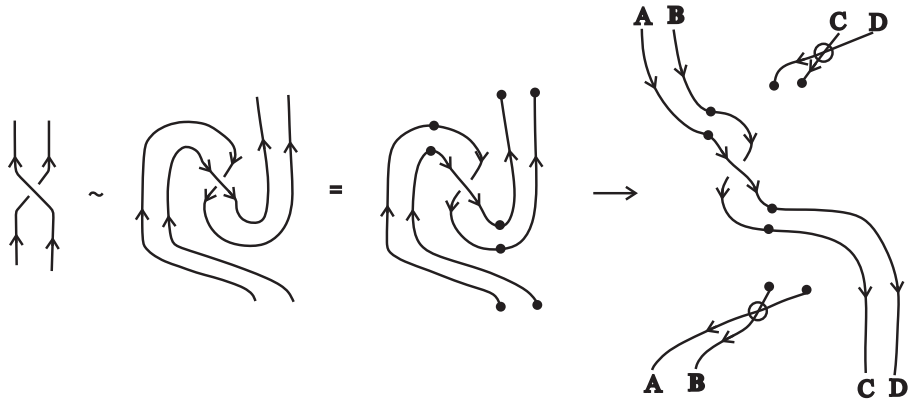}
    \renewcommand{\figurename}{Fig.}
  \caption{Full twist for a classical crossing.}
  \label{Fig.fulltwist}
\end{figure}

\begin{figure}[!htbp]
  \centering
  \includegraphics[width=12cm]{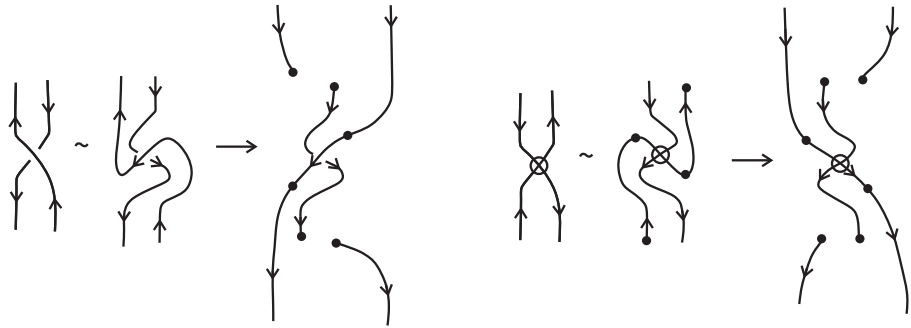}
    \renewcommand{\figurename}{Fig.}
  \caption{Half twist for a classical/virtual crossing.}
  \label{Fig.halftwist}
\end{figure}

We give an example for the braiding process of a virtual link as shown in Fig. \ref{Fig.example}. In the second and third pictures we show a
virtual knot and its preparation for braiding by crossing rotation, respectively.
In the forth picture, we break each over-arc, and name each pair of endpoints with the same letter.
Note that the closure of the virtual braid in Fig. \ref{Fig.example} is not checkerboard colorable.

\begin{figure}[!htbp]
  \centering
  \includegraphics[width=15cm]{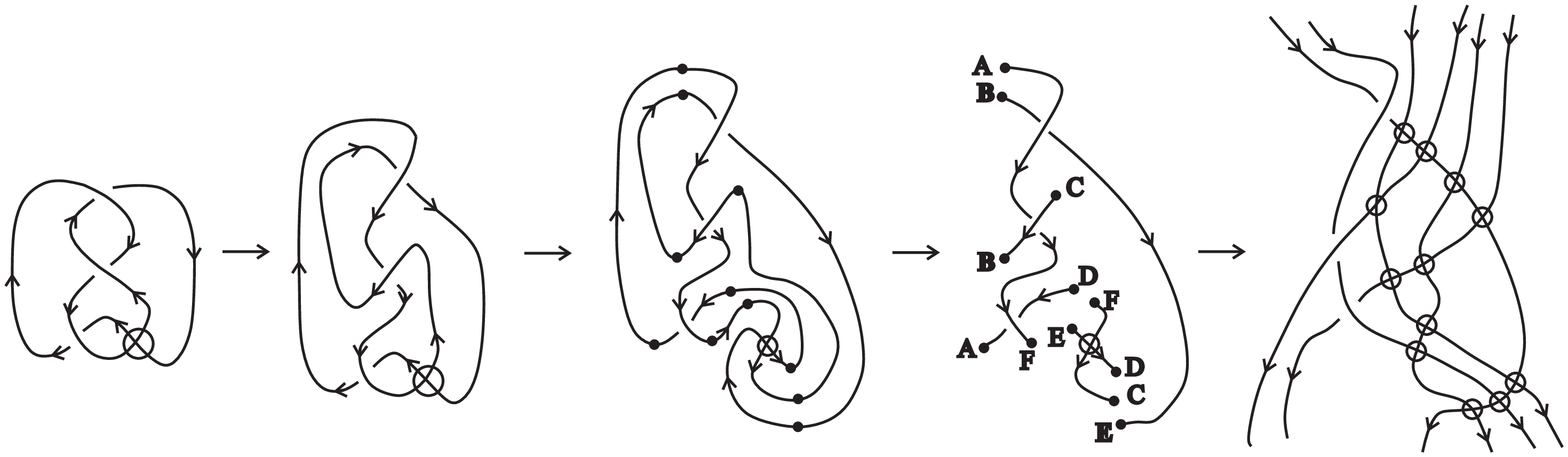}
    \renewcommand{\figurename}{Fig.}
  \caption{Braiding an example.}
  \label{Fig.example}
\end{figure}

\begin{thm}(\cite{KauLam}, Theorem 1)\label{braiding}
Every (oriented) virtual link can be represented by a virtual
braid whose closure is equivalent to the original link.
\end{thm}

\begin{remark}\label{ck}
According to the braiding technique, described in Theorem 1 \cite{KauLam}, which just changes the relative position of classical and virtual crossings by crossing rotation, the original virtual link diagram and the closure of its virtual braid have the same checkerboard colorability.
\end{remark}

Now we state the main result.

\begin{thm}\label{main}
Let $D$ be an oriented checkerboard colorable virtual link diagram.
Then
\begin{enumerate}
  \item[(1)] $AS(D)$ only contains even integer; and
  \item[(2)] for any summand $A^{s}K^{j_{1}}_{i_{1}}K^{j_{2}}_{i_{2}}\cdots K^{j_{v}}_{i_{v}}$ with $1\leq i_1<i_2<\cdots <i_v$, $j_t\geq 1$ for $t=1,2,\cdots,v$, and $v\geq 1$ of $\langle D\rangle_{NA}$, we have $2i_v\leq \sum_{t=1}^{v}i_t\cdot j_t$. In particular, $\langle D\rangle_{NA}$ has no summands like $A^{s}K_{i}$.
\end{enumerate}
\end{thm}

\begin{proof}

By Theorem \ref{braiding} and the above Remark \ref{ck}, we can arrange $D$ as a virtual braid $\mathcal{B}$ with all strands oriented downwards whose closure $\mathcal{\overline{B}}$ is equivalent to $D$.
Note that $\mathcal{\overline{B}}$ is also checkerboard colorable and has the same set of classical crossings as $D$ (see Fig. \ref{Fig.example} for an example).

Assume that $C$ is one checkerboard coloring of $\mathcal{\overline{B}}$. Let $\sigma$ be an (unreduced) state of $\mathcal{\overline{B}}$ without applying reduction rule in Fig. \ref{Fig.reduction}.
According to the coloring $C$, we can obtain a checkerboard coloring $C^{\sigma}$ of the state $\sigma$ where the colorings of arcs are almost the same as the coloring $C$ of $\mathcal{\overline{B}}$ except small segments of arcs around classical crossings (see Fig. \ref{Fig.cbs}). Maxima point and minima point of cusps as shown in Fig. \ref{Fig.cbs} can be defined for any closure of virtual braids (not necessarily checkerboard colorable).

\begin{figure}[!htbp]
\centering
\includegraphics[width=12cm]{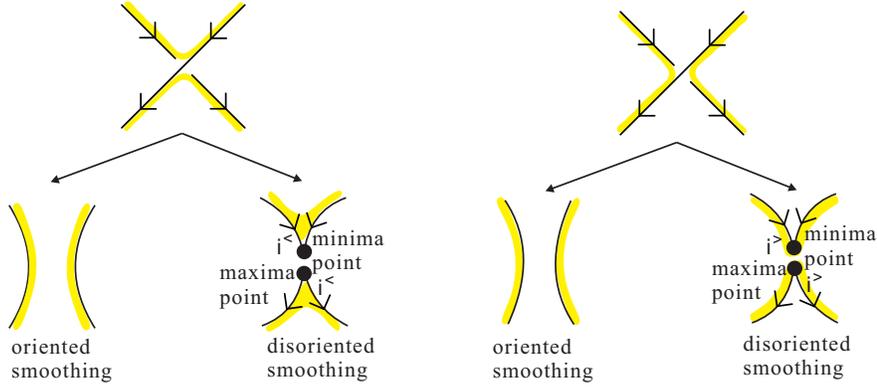}
\renewcommand{\figurename}{Fig.}
\caption{Checkerboard coloring $C^{\sigma}$ of state $\sigma$ around a classical crossing $i$ (the crossing $i$ can also be negative).}\label{Fig.cbs}
\end{figure}

Take a circle graph of the state $\sigma$, the cusp of the circle graph coming from the classical crossing $i$ will be denoted by $i^{<}$ (or $i^{>}$) if the coloring is inside the acute (or obtuse) angle of the cusp. This notation can be defined for any checkerboard colored virtual link diagram. Note that there will be two $i^{<}$'s or two $i^{>}$'s and it will be impossible that one of the pair of cusps of the classical crossing $i$ is $i^{<}$ and the other $i^{>}$ (see Fig. \ref{Fig.cbs}).

We can obtain a \emph{word} from a circle graph if we read cusp in a clockwise direction or a counterclockwise direction and the starting cusp can be chosen at will.
A circle graph is called \emph{trivial} if it does not contain any cusp. The Claim 1 is obvious.

\textbf{Claim 1.} For any circle graph of $\sigma$, it contains maxima point and minima point alternatively. If the circle graph is reduced (but keeping the positions of the remaining cusps), it still holds.

Claim 1 is true for the closure of any virtual braid, not just for the checkerboard colorable case.

\textbf{Claim 2.} If $i^{<}X{i}^{<}Y$ or $i^{>}X{i}^{>}Y$ is a word of a circle graph of $\sigma$, then $X$ and $Y$ are both of even length whether $X$ and $Y$ are reduced or not.

\begin{proof}
Note that one cusp of $i$ has minima point and another has maxima point, so we need go from the minima point to the maxima point or vice versa. By Claim 1 subwords $X$ and $Y$ are both of even length.
\end{proof}

\textbf{Claim 3.} For any reduced circle graph of $\sigma$, it is trivial or the labels of its cusps are all different.

\begin{proof}
If it is not trivial and contains two cusps labelled $i$. Without loss of generality we assume that the two $i$ cusps are both $i^{<}$'s. By Claim 1 the word of the circle graph is either $i^{<}({1}^{>}{2}^{<}\cdots$
${2k-1}^{>}{2k}^{<}){i}^{<}\cdots$ or $i^{<}({1}^{<}{2}^{>}\cdots {2k-1}^{<}{2k}^{>}){i}^{<}\cdots$. In each case, it is not reduced, a contradiction.
\end{proof}

\textbf{Claim 4.} For the state $\sigma$, its $k$-degree is an even integer.

\begin{proof}
Let $\sigma_{0}$ be the state with each classical crossing applying oriented smoothing. The $k$-degree of $\sigma_{0}$ is 0.  The state $\sigma_0$ can be obtained from $\sigma$ by changing the smoothing ways of some classical crossings from disoriented smoothing to oriented smoothing.
It suffices for us to show that the parity of $k$-degree of two states $\sigma_1$ and $\sigma_2$ are the same if $\sigma_2$ can be obtained from $\sigma_1$ by changing the smoothing way of a single classical crossing from disoriented smoothing to oriented smoothing. Let $i$ be the changed classical crossing. Without loss of generality, we assume that the two cusps of the crossing $i$ in $\sigma_1$ are both $i^{<}$'s.

\textbf{Case 1.} Two cusps of $i$ in $\sigma_1$ belong to different circle graphs. In this case two circle graphs of $\sigma_1$ will be merged to a circle graph of $\sigma_2$. There are three subcases.

\textbf{Case 1.1.} Two cusps both appear in the reduced word of the two circle graph. We assume the two reduced words are $i^{<}({1}^{>}{2}^{<}\cdots{2k-1}^{>})$ and ${i}^{<}({2l+1}^{>}{2l+2}^{<}\cdots {2m-1}^{>})$, respectively.
They contribute $K_{k}$ and $K_{m-l}$ to the state $\sigma_1$, respectively. When we only change the smoothing way of $i$, the two circles are merged to one circle with the word $({1}^{>}{2}^{<}\cdots {2k-1}^{>})({2l+1}^{>}{2l+2}^{<}\cdots {2m-1}^{>})$ that will be reduced to ${1}^{>}{2}^{<}\cdots {2(k-m+l)}^{<}$
if $k\geq m-l$ or ${2(k+l)}^{<}{2(k+l)+1}^{>}\cdots {2m-1}^{>}$ if $k<m-l$ (see Fig. \ref{Fig.twotoone1111}).
And the new circle in $\sigma_{2}$ will contribute $K_{|k-m+l|}$ to the state $\sigma_2$. Note that $|k-m+l|$ and $k+m-l$ have the same parity.

\begin{figure}[!htbp]
  \centering
  \includegraphics[width=8cm]{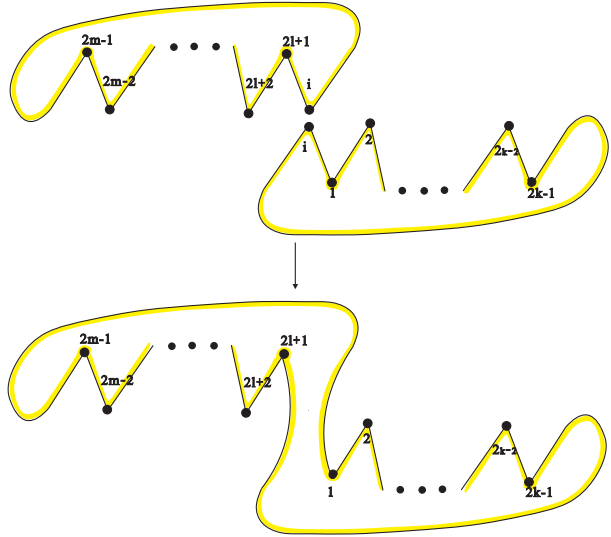}
    \renewcommand{\figurename}{Fig.}
  \caption{Case 1.1.}
  \label{Fig.twotoone1111}
\end{figure}

\textbf{Case 1.2.} Exactly one cusp appear in the reduced word of the two circle graph. We assume the reduced word containing $i$ is $i^{<}({1}^{>}{2}^{<}\cdots{2k-1}^{>})$ and the other word is ${i}^{<}({2l+1}^{<}{2l+2}^{>}\cdots {2m-1}^{<})$.
They contribute $K_{k}$ and $K_{m-l-1}$ to the state $\sigma_1$, respectively. When we only change the smoothing way of $i$, the two circles are changed into one circle with word $({1}^{>}{2}^{<}\cdots {2k-1}^{>})({2l+1}^{<}{2l+2}^{>}\cdots {2m-1}^{<})$, which contributes $K_{k+m-l-1}$ to the state $\sigma_2$.

\textbf{Case 1.3.} Neither of two cusps appear in the reduced word of the two circle graph. This case is similar with Case 1.1, so we leave the proof to the reader.

\textbf{Case 2.} Two cusps of $c$ in $\sigma_1$ belong to a single circle graph. In this case, by Claim 3, the circle graph is not reduced. Note that one circle graph of $\sigma_1$ will be split to two circle graphs of $\sigma_2$. There are two subcases:

\textbf{Case 2.1} The word of the circle graph in $\sigma_1$ is $$i^{<}({1}^{>}{2}^{<}\cdots{2k-1}^{>}{2k}^{<}){i}^{<}({2l+1}^{<}{2l+2}^{>}\cdots {2m}^{>}),$$
\noindent which will be reduced to $i^{<}({1}^{>}{2}^{<}\cdots {2k}^{<})({2l+2}^{>}\cdots {2m}^{>})$.
It contributes $K_{k+m-l}$ to the state $\sigma_1$. When we only change the smoothing way of $i$, the one circle is changed into two circles with words ${1}^{>}{2}^{<}\cdots {2k-1}^{>}{2k}^{<}$ and ${2l+1}^{<}{2l+2}^{>}\cdots {2m-1}^{<}{2m}^{>}$, respectively (see Fig. \ref{Fig.twotoone1}).
And the two new circles in $\sigma_{2}$ will contribute $K_{k}$ and $K_{m-l}$ to the state $\sigma_2$, respectively.

\begin{figure}[!htbp]
  \centering
  \includegraphics[width=8cm]{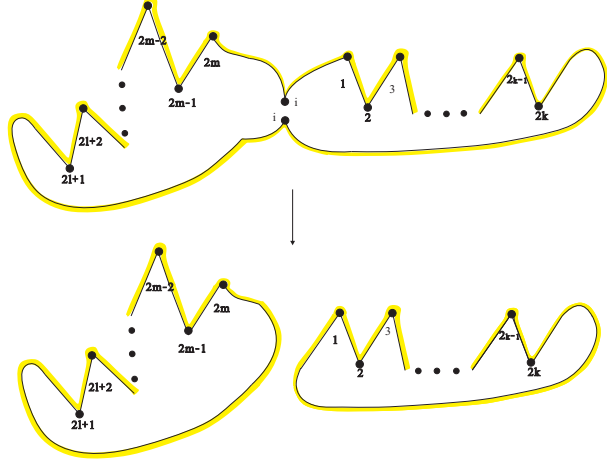}
    \renewcommand{\figurename}{Fig.}
  \caption{Case 2.1.}
  \label{Fig.twotoone1}
\end{figure}

\textbf{Case 2.2} The word of the circle graph in $\sigma_1$ is $$i^{<}({1}^{<}{2}^{>}\cdots {2k-1}^{<}{2k}^{>}){i}^{<}({2l+1}^{<}{2l+2}^{>}\cdots {2m}^{>}).$$
\noindent It will be reduced to $2^{>}\cdots 2(k-m+l)+1^{<}$ if $k\geq m-l$ and $2l+2k+1^{<}\cdots 2m^{>}$ if $k< m-l$, which both contribute $K_{|k-m+l|}$ to the state $\sigma_1$. When we only change the smoothing way of $i$, the one circle is changed into two circles with reduced words ${1}^{<}{2}^{>}\cdots {2k-1}^{<}{2k}^{>}$ and ${2l+1}^{<}{2l+2}^{>}\cdots {2m-1}^{<}{2m}^{>}$, respectively.
And the two new circles in $\sigma_{2}$ will contribute $K_{k}$ and $K_{m-l}$ to the state $\sigma_2$, respectively.
\end{proof}

By Claim 4. Theorem \ref{main} (1) holds. Now we prove Theorem \ref{main} (2).

Suppose that $\langle D\rangle_{NA}$ contains a summand with $K_{i_{1}}^{j_{1}}K_{i_{2}}^{j_{2}}\cdots K_{i_{v}}^{j_{v}}$ ($1\leq i_1<i_2<\cdots <i_v$ and $j_t\geq 1$ for $t=1,2,\cdots,v$), and the state $\sigma$ contains $j_1$ $K_{i_{1}}$ reduced circle graphs, $j_2$ $K_{i_{2}}$ reduced circle graphs, $\cdots$, and $j_v$ $K_{i_{v}}$ reduced circle graphs. Considering a $K_{i_{v}}$ reduced circle graph, then, by Claim 3, labels of all cusps of the circle are all different.

Without loss of generality, we assume the word of the reduced circle graph is $${1}^{<}{2}^{>}\cdots {2i_{v}-1}^{<}{2i_{v}}^{>}$$ and ${2i-1}^{<}$ and ${2i}^{>}$ ($1\leq i\leq i_{v}$) correspond to maxima point and minima point,  respectively. We claim that there must have $i_{v}$ survived cusps with minima points and $i_{v}$ survived cusps with maxima points in other reduced circle graphs of the state $\sigma$.

Otherwise if the cusp ${2i-1}^{<}$ ($1\leq i\leq i_{v}$) with minima point cancels with some cusp $j^{<}$ with maxima point, then the cusp ${j}^{<}$ with
minima point may survive or be cancelled. If it survives, then we show that it is our needed survived cusp with minima point by proving it can not be in the reduced circle graph $K_{i_{v}}$.
If ${j}^{<}$ is in the reduced circle graph $K_{i_{v}}$, then ${j}$ should be an even number since it has the minima point.
But since ${j}$ has the superscript $^{<}$, then $j$ should be an odd number by the word of reduced circle graph $K_{i_{v}}$, a contradiction.
If it is cancelled, we continue this process until we obtain a survived cusp with minima point.
\end{proof}

\section{Applications}
\noindent

In \cite{Ima}, Imabeppu compared the two checkerboard colorability criteria based on Sawollek polynomials with other three known criteria, and proved that these five criteria are all independent.
By using these criteria Imabeppu determined checkerboard colorability of virtual knots up to four
classical crossings, with seven exceptions (that is, virtual knots $4.55, 4.56, 4.59, 4.72, 4.76, 4.77, 4.96$).

We use Theorem \ref{main} to detect the checkerboard colorability of the remaining seven virtual knots. The normalized arrow polynomials of virtual knots $4.55$, $4.56$, $4.59$, $4.72$, $4.76$, $4.77$, $4.96$ are as follows.

$$\langle4.55\rangle_{NA}=A^4+A^{-4}+1-(A^4+A^{-4}+2)K_1^{2}+2K_2,$$

$$\langle4.56\rangle_{NA}=A^4 \left(-\left({K_1}^2-1\right)\right)+\frac{1-{K_1}^2}{A^4}-2 {K_1}^2+2 {K_2}+1,$$

$$\langle4.59\rangle_{NA}=\frac{A^8 \left({K_2}-{K_1}^2\right)+A^4 \left(3-2 {K_1}^2\right)-{K_1}^2+{K_2}}{A^4},$$

$$\langle4.72\rangle_{NA}=1,$$

$$\langle4.76\rangle_{NA}=\frac{A^8 \left({K_2}-{K_1}^2\right)+A^4 \left(3-2 {K_1}^2\right)-{K_1}^2+{K_2}}{A^4},$$

$$\langle4.77\rangle_{NA}=\frac{A^8 \left({K_2}-{K_1}^2\right)+A^4 \left(3-2 {K_1}^2\right)-{K_1}^2+{K_2}}{A^4},$$

$$\langle4.96\rangle_{NA}=\frac{{K_1}^2}{A^6}+A^4 ({K_3}-{K_1} {K_2})-A^2 \left({K_1}^2-1\right)-{K_1} {K_2}+{K_1}.$$

Note that the arrow polynomials of virtual knots $4.55, 4.56, 4.59$, $4.76, 4.77$ containing a summand with $K_2$, $4.96$ containing a summand with $K_3$ are against Theorem \ref{main} (2), and virtual knot $4.96$ is also against Theorem \ref{main} (1). So they are all not checkerboard colorable. It is unknown to us whether virtual knot $4.72$ is checkerboard colorable or not.

\section{Acknowledgements}
This work is supported partially by Hu Xiang Gao Ceng Ci Ren Cai Ju Jiao Gong Cheng-Chuang Xin Ren Cai (No. 2019RS1057).
Deng is also supported by Doctor's Funds of Xiangtan University (No. 09KZ$|$KZ08069).
Jin is also supported by NSFC (No. 11671336) and the Fundamental Research
Funds for the Central Universities  (No. 20720190062).
Kauffman is also supported by the Laboratory of Topology and Dynamics, Novosibirsk State University (contract no. 14.Y26.31.0025 with the Ministry of Education and Science of the Russian Federation).

\end{document}